\documentclass{amsart}
\usepackage{amscd,amsmath,amsthm,amssymb}
\usepackage[utf8]{inputenc}
\usepackage[dvipsnames]{xcolor}
\usepackage{enumerate}
\usepackage{graphicx}
\usepackage{tikz}
\usepackage{cleveref}
\usepackage{float}
\usepackage{mathrsfs}
\usepackage{amsaddr}
\usepackage{csquotes}
\usepackage{mathtools}

 \def\NN{{\mathbb N}}
 \def\ZZ{{\mathbb Z}}
 \def\QQ{{\mathbb Q}}
 \def\RR{{\mathbb R}}

 \def\J{{\mathcal J}}
 \def\C{{\mathcal C}}
 \def\B{{\mathcal B}}
 
 \def\iso{\cong}
 \def\rank{\operatorname{rank}}
 \def\reg{\operatorname{reg}}
 \def\sort{\operatorname{sort}}
 
 \def\supp{\operatorname{supp}}

\newcommand{\precdot}{\prec\mathrel{\mkern-3mu}\mathrel{\cdot}}
\newcommand{\succdot}{\mathrel{\cdot}\mathrel{\mkern-3mu}\succ}

 \newtheorem{Theorem}{Theorem}[section]
 \newtheorem{Lemma}[Theorem]{Lemma}
 \newtheorem{Corollary}[Theorem]{Corollary}
 \newtheorem{Proposition}[Theorem]{Proposition}
 
 \theoremstyle{definition}
 \newtheorem{Remark}[Theorem]{Remark}
 \newtheorem{Example}[Theorem]{Example}
 \newtheorem{Definition}[Theorem]{Definition}
 \newtheorem{Problem}[Theorem]{Problem}

\begin{document}

\title{Chain algebras of finite distributive lattices}
\author{Oleksandra Gasanova\textsuperscript{*}, Lisa Nicklasson\textsuperscript{\dag}}
\address{\textsuperscript{*}Universität Duisburg-Essen \\ \textsuperscript{\dag}Mälardalens universitet}
\email{oleksandra.gasanova@uni-due.de}
\email{lisa.nicklasson@mdu.se}

\thanks{}



\begin{abstract}
We introduce a family of toric algebras defined by maximal chains of a finite distributive lattice. Applying results on stable set polytopes we conclude that every such algebra is normal and Cohen-Macaulay, and give an interpretation of its Krull dimension in terms of the combinatorics of the underlying lattice.
When the lattice is planar, we show that the corresponding chain algebra is generated by a sortable set of monomials and is isomorphic to a Hibi ring of another finite distributive lattice. As a consequence it has a defining toric ideal with a quadratic Gröbner basis, and its $h$-vector counts ascents in certain standard Young tableaux. If instead the lattice has dimension $n>2$, we show that the defining ideal has minimal generators of degree at least $n$.
\end{abstract}

\maketitle

\section*{Introduction}
    In this note we study toric algebras arising from maximal chains of finite distributive lattices. Toric rings defined by combinatorial objects is an emerging area of research, motivated by its thriving interplay between toric geometry, combinatorics, and commutative algebra. Edge rings of graphs (\cite[Section 5]{binomialideals}, \cite{OH_graph}, \cite{SVV}), base rings of matroids (\cite{Blasiak}, \cite{White}), Hibi rings of posets (\cite{Enesurvey}, \cite{hibioriginal}), and rings of stable set polytopes (\cite{MOS}) are some examples of such algebras. In the present paper we introduce a new class of algebras of this type, namely those defined by maximal chains of finite distributive lattices. 

Given a finite distributive lattice $L$, we identify its elements with variables in a polynomial ring over a field $K$. Each maximal chain of elements in $L$ is then associated with a monomial, and we let $K[\C_L]$ denote the \emph{chain algebra} generated by these monomials. The goal is then to explore the properties of the chain algebra in connection to the combinatorial properties of $L$. When $L$ is planar it turns out that the algebra is isomorphic to a Hibi ring defined by an interval of Young's lattice. A connection to stable set polytopes enables us to draw conclusions in the general case, in particular it follows that chain algebras are normal and Cohen-Macaulay. Moreover, a neat formula for the dimension of a chain algebra is obtained. 


This paper is organised as follows. Section~1 reviews necessary background on lattices, graphs, convex polytopes, Hibi rings, and sortable sets of monomials.

In Section~2 we introduce the notion of a chain algebra $K[\C_L]$ of a finite distributive lattice $L$ and explain the connection of such algebras to stable set polytopes. Indeed, it is not hard to see that any chain algebra is the toric algebra of a specific face of the stable set polytope of the co-comparability graph of $L$. This implies normality and Cohen-Macaulayness of any chain algebra and allows to bound from above the degrees of minimal generators of its defining ideal. We also bound from below the maximal degree of a minimal generator of this ideal, and compute the Krull dimension of a chain algebra in terms of $L$. 

Section~3 is dedicated to the study of chain algebras of planar finite distributive lattices. The highlight of this paper is  \Cref{thm:mainplanar} which says that the following are equivalent for a finite distributive lattice $L$:

\begin{enumerate}
    \item $L$ is planar,
    \item $K[\C_L]$ can be generated by a sortable set of monomials,
    \item The defining ideal of $K[\mathcal{C}_L]$ has a quadratic Gröbner basis with respect to DegRevLex,
    \item The defining ideal of $K[\mathcal{C}_L]$ is quadratically generated,
    \item $K[\C_L]$ is Koszul,
    \item $K[\C_L]$ is a Hibi ring.

\end{enumerate}

Since the chain algebra of a  planar finite distributive lattice $L$ can be seen as a Hibi ring of another distributive lattice $L'$, it is then natural to talk about $L'$ itself. It turns out that elements of $L'$ are maximal chains of $L$, and that $L'$ can in fact be identified with an interval of Young's lattice. In addition, we discuss a combinatorial interpretation of the Hilbert series and an alternative interpretation of the Krull dimension formula for the planar case.

Inspired by the connection to Hibi rings and sortable sets of monomials, in Section~4 we take a detour 
to closely study the connection between Hibi rings and sortable sets of monomials. 

\section{Background}\label{basic}
In this section we provide a quick review of necessary background on posets and lattices, graph theory, integral convex polytopes, Hibi rings, and sortable sets of monomials. 

\subsection{Posets and lattices} Here we will review the definitions and basic properties of posets and lattices. For further reading we recommend 
\cite{Stanley}.
A \emph{partially ordered set} $P$ (or \emph{poset}, for short) is a set, together with a binary relation $\preceq$ (or $\preceq_P$ when there is a possibility of confusion), 
satisfying the axioms of reflexivity, antisymmetry and transitivity. 
In this paper all the posets will be finite.

An \emph{induced subposet} of $P$ is a subset $Q$ of $P$ endowed with a partial order $\prec_Q$ such that for $u_1,u_2,\in Q$ we have $u_1 \prec_Q u_2$ if and only if $u_1 \prec_P u_2$.
Two posets $P$ and $Q$ are \textit {isomorphic}, denoted $P\iso Q$, if there exists an order preserving bijection $\phi: P\rightarrow Q$, whose inverse is also order preserving, that is,
$u_1 \preceq_P u_2$  $\Leftrightarrow$ $ \phi(u_1)\preceq_Q \phi(u_2)$.
For $u_1,u_2 \in P$ we say that $u_1$ \textit{covers} $u_2$ (or $u_2$ \textit {is covered} by $u_1$) denoted $u_2\precdot u_1$,
if $u_2 \prec u_1$ and no element $u\in P$ satisfies $u_2 \prec u \prec u_1$.

A \emph{chain} of $P$ is an induced subposet in which any two elements are comparable.
An \emph{antichain} of $P$ is an induced subposet in which no two elements are comparable. A chain of $P$ is called \textit{maximal} if it is not contained in a larger chain of $P$.
The \textit{width} of a poset is the number of elements in the largest antichain of $P$. By Dilworth's theorem (\cite[Theorem~1.1]{Dil}), it is also the smallest number of disjoint chains needed to cover $P$.
If every maximal chain of a finite poset $P$ has the same length $r$, then we say that $P$ is \textit{pure}. In this case there is
a {unique} \emph{rank function} $\rank: P \rightarrow \{0,1,\ldots,r\}$ such that $\rank(u) = 0$ if $u$ is a minimal
element of $P$, and $\rank(u_2) = \rank(u_1) + 1$ if $u_1\precdot u_2$ in $P$. 
We will also say that $P$ has rank $r$. 
For an arbitrary poset $P$ its rank is defined as the length of its longest chain.

For $u_1, u_2 \in P$ a \emph{join} of $u_1$ and $u_2$ is an element $a\in P$ such that $a\succeq u_1$ and $a\succeq u_2$, and such that any other element $a'$ with the same property satisfies $a'\succeq a$. If a join of $u_1$ and $u_2$ exists, it is clearly unique, and is denoted by $u_1\vee u_2$.  Dually, one can define the \textit{meet} $u_1\wedge u_2$, when it exists. A \textit{lattice}
is a poset for which every pair of elements has a meet and a join. An element of a finite lattice is called \textit{join-irreducible} if it covers exactly one element, and \emph{meet-irreducible} if it is covered by exactly one element.
 A lattice is called \emph{distributive} if the operations $\vee$ and $\wedge$ distribute over each other.

An \emph{ideal} of a poset $P$ is a subset $I$ of $P$ such that if $u_1 \in I$ and $u_2\preceq u_1$, then $u_2 \in I$. In other words, an ideal is a downward closed subset of $P$. The set $\mathcal{J}(P)$ of ideals of $P$ is a poset, ordered by inclusion. Since the union and intersection of ideals is again an ideal, it follows that $\mathcal{J}(P)$ is a lattice with the join of two ideals being their union, and the meet being their intersection. From the well-known distributivity of set union and intersection over
one another it follows that $\mathcal{J}(P)$ is in fact a distributive lattice. Birkhoff's fundamental theorem for
finite distributive lattices states that the converse is true when $P$ is finite.
\begin{Theorem}[Birkhoff]
Let $L$ be a finite distributive lattice. Then $L\iso \J(P)$, where $P$ is the induced subposet of join-irreducible elements of $L$.
\end{Theorem}

The \emph{dimension} of a lattice $L\cong \J(P)$ is defined to be the width of $P$. This is well defined as $\J(P) \cong \mathcal{J}(Q)$ if and only if $P\iso Q$ (\cite[Proposition~3.4.2]{Stanley}). 
We say that $L$ is \emph{planar} if $\dim (L)\le 2$.  

It is also useful to keep in mind the following:
\begin{Theorem}[{\cite[Proposition~3.4.5]{Stanley}}]
If $P$ is a finite poset, then $\mathcal{J}(P)$ is pure of rank $|P|$. The rank of an element of $\mathcal{J}(P)$ is its cardinality as
an ideal of $P$.
\end{Theorem}

\subsection{Some graph theory}
Let $G=(V,E)$ be a finite simple graph. An \emph{induced subgraph} $H=(V',E')$ of $G$ is a graph with the vertex set $V' \subseteq V$ and the edge set $E' \subseteq E$ such that every edge in $E$ with both ends in $V'$ must also be an edge of $E'$.  
An \emph{induced cycle} of graph $G$ is an induced subgraph which is a cycle. A \emph{perfect matching} of a graph is a set of edges such that each vertex is contained in exactly one edge from this set.

A subset $W\subseteq V$
is called \textit{stable} (or \textit{independent}) if $\{i, j\}\not\in E$ for all $i, j \in W$. In other words, a stable set in $G$ is $W\subseteq V$ such that no pair of vertices from $W$ are adjacent. Dually, a \emph{clique} in $G$ is $W\subseteq V$ such that every pair of different vertices from $W$ are adjacent.

 We say that a graph $G=(V,E)$ is \emph{bipartite} if $V$ can be partitioned into two stable sets, or equivalently if $G$ has no induced cycle of odd length.

A  \emph{colouring} of $G$ is a labelling of $V$ with colours such that no two adjacent vertices have the same colour. 
The smallest number of colours needed to colour $G$ is called its \emph{chromatic number}.
If for any induced subgraph $H \subset G$ the chromatic number of $H$ equals the size of the largest clique in $H$, then we say that $G$ is a \emph{perfect graph}.

Let $P$ be a poset with the underlying set $\{t_1,\ldots,t_s\}$. The (co-)comparability graph of $P$ is a graph $G$ with $V(G)=\{1,\ldots,s\}$ and $\{i, j\}\in E(G)$ if and only if $i\not=j$ and $t_i$ and $t_j$ are (in)comparable. According to Dilworth's theorem \cite{Dil}, any co-comparability graph is perfect, which we will use later in the paper.

\subsection{Integral convex polytopes}
Let $Q\subset \RR^s$ be an integral convex polytope, that is, a convex polytope whose vertices have integer
coordinates, and  let $Q \cap \ZZ^s =\{ \mathbf a_1, \ldots, \mathbf a_m\}$. Let $K[X, X^{-1}, t] = K[x_1, x^{-1}_1, \ldots, x_s, x^{-1}_s, t]$. Given an integer vector $\mathbf a = (a_1, \ldots a_s) \in \ZZ^s$, we
set $X^\mathbf a t = x^{a_1}_1\cdots x^{a_s}_s\cdot t \in K[X, X^
{-1},t]$. Then, the \textit{ toric algebra} of $Q$ is the subalgebra $K[Q]$ of $K[X, X^{-1}, t]$
generated by $\{X^{\mathbf a_1}t,\ldots, X^{\mathbf a_m}t\}$ over $K$. The toric ideal $I_Q$ of $Q$ is the kernel of a surjective homomorphism 
$\pi : K[y_1, \ldots, y_m] \to K[P]$ defined by $\pi(y_i) = X^{\mathbf a_i}t$ for $1 \le i \le m$.

An integral convex polytope is called \emph{compressed} if all of its pulling triangulations
are unimodular (see \cite{stanley_polytope} for details). Equivalently (\cite{Sturmfels}), a compressed polytope is an integral convex polytope $Q$ such that the initial ideal of $I_Q$ with respect to any reverse lexicographic monomial order on $K[y_1, y_2,\ldots ,y_m]$ is generated by squarefree monomials. If $Q$ is compressed, then all faces of $Q$ are again
compressed. 

\subsection{Hibi rings}
Let $L$ be a finite lattice and $S=K[x_a\mid a\in L]$ a polynomial ring over a field $K$. The \emph{Hibi ring} of $L$, denoted $K[L]$, is the quotient ring $S/I$ where $I$ is the ideal generated by all binomials $x_ax_b-x_{a\wedge b}x_{a\vee b}$ for $a, b \in L$. These rings were introduced by Hibi in \cite{hibioriginal}, where he proves that $K[L]$ is toric if and only if $L$ is distributive, and in this case $K[L]$ is a normal Cohen-Macaulay domain. All Hibi rings considered in this paper will be defined by distributive lattices.  

Another result from \cite{hibioriginal} is that the Hibi ring of a finite distributive lattice $L=\J(P)$ is Gorenstein if and only if $P$ is pure. It is also known that $\reg K[L]=|P|-\rank(P)-1$ and $\dim(K[L])=|P|+1$. For these and more results on Hibi rings the reader may consult the survey \cite{Enesurvey}.

Let us index the elements of $P=\{p_1, \ldots, p_d\}$ so that $p_i\prec p_j$ implies $i<j$. In other words, we are giving $P$ a natural labelling.  Let $\mathcal{L}(P)$ be the set of linear extensions of $P$, meaning the set of bijections $w: P \to \{1,2, \ldots, d\}$  respecting the partial order on $P$. A \emph{descent} of $w$ is an index $i$ such that $w(p_j)=i$ and $w(p_k)=i+1$ where $j>k$, and $\operatorname{des}(w)$ denotes the number of descents of $w$. Let $\Omega_P(m)$ denote the function counting the number of order-preserving maps $P\to \{0,\ldots, m\}$. This function is in fact a polynomial of degree $|P|$ and is called \textit{the order polynomial} of $P$.  With the natural labelling we fixed above and using equation (3.65) in Stanley's book \cite {Stanley} one obtains (note that Stanley uses a slightly different definition for $\Omega_P(m)$):

\begin{equation}\label{eq:Hilbert_series}
\sum_{m\ge 0}\Omega_P(m)z^m=\frac{\sum_{w \in \mathcal{L}(P)}z^{\operatorname{des}(\omega)}}{(1-z)^{d+1}}.
\end{equation}
It is pointed out in  \cite[Proposition~2.3]{hibihilbert} that the expression \eqref{eq:Hilbert_series} equals the Hilbert series of the Hibi ring of a $L=\J(P)$. As $K[L]$ is a Cohen-Macaulay domain, the $h$-vector is symmetric if and only if the ring is Gorenstein. 

\subsection{Sortable monomial sets}
Let $r$ be a positive integer and let  $S_r$ be the
$K$-vector space spanned by the
monomials of degree $r$ in the standard graded polynomial ring $S = K[t_1,\ldots,t_s]$. Take two monomials $u_1,u_2\in
S_r$. We write $u_1u_2 = t_{i_1}t_{i_2}\cdots t_{i_{2r}}$ with $1\le i_1\le i_2\le\ldots\le i_{2r} \le s$ and define
$$v_1=t_{i_1}t_{i_3}\cdots t_{i_{2r-1}},\quad v_2= t_{i_2}t_{i_4}\cdots t_{i_{2r}}.$$
The pair $(v_1, v_2)$ is called the \emph{sorting} of $(u_1, u_2)$. Sorting pairs of monomials in this way defines a map
$\sort: S_r \times S_r \to S_r \times S_r$.  
A subset $B \subseteq S_r$ of monomials is called \emph{sortable} if
$\sort(B \times B) \subseteq B \times B$. A pair $(u_1, u_2)$ is called \emph{sorted} if $\sort(u_1, u_2)=(u_1, u_2)$. Since $\sort(u_1,u_2)=\sort (u_2,u_1)$, we will sometimes consider unordered pairs and say that $\{u_1,u_2\}$ is sorted if either $(u_1, u_2)$ or $(u_2, u_1)$ is a sorted pair.

Let $B=\{u_1,\ldots, u_m\} \subset S_r$ be a sortable set of monomials and let $K[B]$ be the algebra
generated by $B$. Let $R=K[T_1,\ldots, T_m]$ and let $\phi : R \to K[B]$ be
the surjective $K$-algebra homomorphism defined by $T_i\mapsto u_i$ for all $i=1,\ldots, m$. Then $K[B] \cong R/I_B$ where $I_B = \ker(\phi)$.

A monomial order $>$ on $R$ is called a \emph{sorting order} if $T_iT_j>T_kT_{\ell}$ when $\sort(u_i,u_j)=(u_k,u_{\ell})$. 

\begin{Theorem}[{\cite[Theorem 14.2]{Sturmfels}}]\label{thm:Sturmfels}
\label{thm:order}
Let $B=\{u_1,\ldots, u_m\}$ be a sortable set of monomials and let
\begin{equation}\label{eq:sort_binomial}
G = \left\{\underline{T_{i}T_{j}}-T_{k}T_{\ell}: \{u_i, u_j\} \text{ not sorted}, (u_k, u_{\ell}) = \sort(u_i, u_j)\right\} \subset R.
\end{equation}
Then a sorting order on $R$ exists, and $G$ is a reduced Gröbner basis of $I_B$ w.\,r.\,t.\ any sorting order.
\end{Theorem}

\begin{Remark}
In \cite{Sturmfels} Sturmfels constructs a monomial order and proves {\cite[Theorem 14.2]{Sturmfels}} with respect to this particular order. However, the proof only relies on the fact that the monomial order is a sorting order as defined above. We have phrased Theorem \ref{thm:Sturmfels} differently from \cite{Sturmfels}, in order to emphasise that there might exist several sorting orders and the result holds for any of them.
\end{Remark}

A sorting order can often be realised as a degree reverse lexicographic order, for short DegRevLex. As this order will be employed later, we recall its definition here. Ordering the variables of the polynomial ring $R=K[T_1, \ldots, T_m]$ as $T_1> \dots > T_m$ we have 
\[
T_1^{\alpha_1} \cdots T_m^{\alpha_m}<T_1^{\beta_1} \cdots T_m^{\beta_m} \quad \text{according to DegRevLex}
\]
if $\alpha_1 + \dots + \alpha_m< \beta_1 + \dots + \beta_m$ or $\alpha_1 + \dots + \alpha_m= \beta_1 + \dots + \beta_m$ and there is an $i$ such that $\alpha_i>\beta_i$ and $\alpha_j=\beta_j$ for any $j=i+1,\ldots, m$.  Note that reordering the variables, and reordering the comparison of exponents accordingly, gives rise to another DegRevLex order. In total, there are $m!$ possible DegRevLex orders on $R$.

\section{The chain algebra of a finite distributive lattice}
\label{sec:chain}

\subsection{Definition and connection to stable set polytopes}

Let $L$ be a finite distributive lattice (FDL for short) of rank $r$. We label the elements of $L$ by $t_1, \ldots, t_s$ so that $t_1$ is the unique minimal element, and $t_{s}$ is the unique maximal element. We write $t_i\preceq t_j$ to indicate the inclusion of the respective ideals.
Recall that maximal chains $C=\{t_1=t_{i_0}\precdot t_{i_1} \precdot \ldots \precdot t_{i_{r}}=t_{s}\}$ of $L$ always have cardinality $r+1$, start at $t_1$, which is the unique element of rank $0$, and end at $t_s$, which is the unique element of rank $r$.

To the maximal chain $C$ we associate the monomial  $t_{C}:=t_{i_0}\!\cdots t_{i_{r}}$ in the polynomial ring $K[t_1,\ldots, t_{s}]$ over some field $K$.

\begin{Definition}
The \textit{chain algebra} of an FDL $L$ is the $K$-subalgebra of $K[t_1,\ldots, t_s]$ generated by all monomials $t_C$, where $C$ is a maximal chain of $L$. This algebra is denoted by $K[{\C}_L]$. 
\end{Definition}
Clearly, the properties of $K[{\C}_L]$  are independent of the labelling of $L$, and we may choose a labelling which fits our purpose. 


\begin{Example}
\label{work example}
Let $P$ and $L=\mathcal{J}(P)$ be as described in \Cref{chain algebra example}. Here $r=4$ and $s=8$. Then 
$$K[\mathcal{C}_L]=K[t_1t_2t_4t_6t_8,\, t_1t_2t_4t_7t_8, \, t_1t_3t_4t_6t_8, \, t_1t_3t_4t_7t_8, \, t_1t_3t_5t_7t_8]\subset K[t_1,\ldots, t_8].$$
One can verify with the help of computer algebra that
$$K[\mathcal{C}_L]\iso K[T_1,T_2,T_3,T_4,T_5]/(T_2T_3-T_1T_4).$$

\begin{figure}
\centering
\begin{tikzpicture}[scale=1.4]

\draw [thick]    (-1,0) -- (-1,1);
\draw [thick]    (-1,1) -- (0,0);
\draw [thick]     (0,0) -- (0,1);

\draw [thick]    (3.5,-1) -- (3.5+cos 45,-1+sin 45) -- (3.5+2*cos 45,-1+2*sin 45)--(3.5+cos 45,-1+3*sin 45)-- (3.5,-1+4*sin 45)--(3.5-cos 45,-1+3*sin 45)--(3.5,-1+2*sin 45)--(3.5+cos 45,-1+3*sin 45)--(3.5-cos 45,-1+sin 45)--(3.5,-1) -- (3.5+cos 45,-1+sin 45)--(3.5,-1+2*sin 45);

\node [below,left, thin] at (-1,0) {$a$};
\node [above,left, thin] at (-1,1) {$c$};
\node [below,right, thin] at (0,0) {$b$};
\node [below,right, thin] at (0,1) {$d$};

\node [below, thin] at (3.5,-1) {$t_1=\emptyset$};
\node [below, right,thin] at (3.5+cos 45,-1+sin 45) {$t_3=\{b\}$};
\node [right,thin] at (3.5+2*cos 45,-1+2*sin 45) {$t_5=\{b,d\}$};
\node [above,right,thin] at (3.5+cos 45,-1+3*sin 45) {$t_7=\{a,b,d\}$};
\node [above,thin] at (3.5,-1+4*sin 45) {$t_8=\{a,b,c,d\}$};
\node [left,thin] at (3.5-cos 45,-1+3*sin 45) {$t_6=\{a,b,c\}$};
\node [left,thin] at (3.5,-1+2*sin 45) {$t_4=\{a,b\}$};
\node [left,thin] at (3.5-cos 45,-1+sin 45) {$t_2=\{a\}$};

\fill[fill=black] (3.5,-1) circle (0.03 cm);
\fill[fill=black] (3.5+cos 45,-1+sin 45) circle (0.03 cm);
\fill[fill=black] (3.5+2*cos 45,-1+2*sin 45) circle (0.03 cm);
\fill[fill=black] (3.5+cos 45,-1+3*sin 45) circle (0.03 cm);
\fill[fill=black] (3.5,-1+4*sin 45) circle (0.03 cm);
\fill[fill=black] (3.5,-1+2*sin 45) circle (0.03 cm);
\fill[fill=black] (3.5-cos 45,-1+3*sin 45) circle (0.03 cm);
\fill[fill=black] (3.5-cos 45,-1+sin 45) circle (0.03 cm);

\fill[fill=black] (0,0) circle (0.03 cm);
\fill[fill=black] (-1,0) circle (0.03 cm);
\fill[fill=black] (-1,1) circle (0.03 cm);
\fill[fill=black] (0,1) circle (0.03 cm);

\end{tikzpicture}
\caption{A poset $P$ and the lattice $\J(P)$}
\label{chain algebra example}
\end{figure}
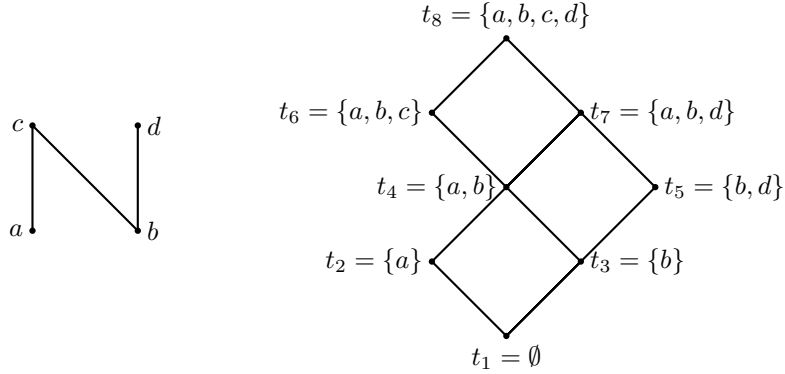
\end{Example}


We will now proceed to explaining the connection of chain algebras to stable set polytopes, which will yield some properties of chain algebras. To this end, we are interested in a specific class of polytopes, namely the so-called \emph{stable set polytopes} of finite simple graphs.

Let $G=(V,E)$ be a graph with $V=[s]=\{1,\ldots,s\}$. To a subset
$W \subseteq [s]$, we associate the 
$(0,1)$-point $\rho(W) = \sum_{j\in W}
 {e_j}\in \RR^s$. Here, $e_i$
is the $i$th standard basis vector of $\RR^s$.  Let $S(G)$ denote the set of all stable sets of $G$. The \textit{stable set polytope} $Q_{G}\subset \RR^s$ of a simple graph $G$ is the convex hull of $\{\rho(W) : W\in S(G)\}$.
Note that $Q_{G}$ is an $s$-dimensional polytope since the empty set and each singleton of $[s]$ are stable sets. Also note that  $Q_{G}$ is a $(0,1)$-polytope (the convex hull of $(0,1)$-points). This implies that all $\rho(W)$ are vertices of $Q_{G}$ and that $Q_{G}$ cannot contain other points in $\ZZ^s$ apart from its vertices. It is also known (see \cite{OH_polytope}) that if $G$ is a perfect graph, then $Q_G$ is compressed.

In order to connect stable set polytopes to chain algebras, let $G_L$ be the co-comparability graph of a lattice $L=\J(P)$.  Then (maximal) independent sets of $G_L$ correspond to (maximal) chains of $L$.
Let $F_{G_L}$ be the convex hull of $\{\rho(W)\mid W \text { is a maximal independent set of } G_L\}$. Then $F_{G_L}$ is a face of the stable set polytope $Q_{G_L}$. Indeed, the inequality $x_1+x_2+\dots+x_{s}\le r+1$, where $r$ is the rank of $P$, is valid for the whole polytope $Q_{G_L}$. Since equality is attained precisely at the points of $F_{G_L}$, the claim follows.


Note that $F_{G_L}$, being a $(0,1)$-polytope itself, does not contain any other points in $\ZZ^s$ apart from those coming from maximal chains of $L$. Therefore, we conclude the following:

\begin{Proposition}\label{prop:normal_CM}
 Let $L$ be an FDL, and $F_{G_L}$ the polytope defined above. Then $K[\C_L] \cong K[F_{G_L}]$. Moreover, the initial ideal of $I_{F_{G_L}}$ w.\,r.\,t.\ any DegRevLex order is squarefree, and as a consequence $K[\C_L]$ is normal and Cohen-Macaulay.   
\end{Proposition}

\begin{proof}
The claimed isomorphism follows from the discussion above. For the rest, consider the following chain of implications: 

$G_L$ is a co-comparability graph (of $L$) $\xRightarrow {\text{\cite{Dil}}}$ $G_L$ is perfect $\xRightarrow{\text{\cite{OH_polytope}}}$ $Q_{G_L}$ is a compressed polytope $\xRightarrow{\text{}}$ $F_L$ is a compressed polytope $\xRightarrow{\text{\cite{Sturmfels}}}$  the initial ideal of $I_{F_L}$ is generated
by squarefree monomials with respect to \emph{any} DegRevLex $\xRightarrow{\text{\cite{hochster},\cite{Sturmfels}}}$ the corresponding toric ring $K[F_L]\iso K[\C_L]$ is normal and Cohen-Macaulay. 
\end{proof}

We will see in \Cref{cor:upper_deg_bound} that having a squarefree initial ideal with respect to any DegRevLex order gives an upper bound for the degree of a minimal generator of $I_{F_L}$.

\subsection{Krull dimension and bounds on degrees of generators}

The next result concerns the Krull dimension of a chain algebra. 
\begin{Theorem}
\label{thm:dim}
Let $P$ be a poset and let $L=\mathcal{J}(P)$. Then $\dim K[\mathcal{C}_L]=|L|-|P|$.
\end{Theorem}
\begin{proof}
Let $G(L)$ be the graph constructed as follows: the vertices of $G(L)$ are the elements of $L$; $t_i$ and $t_j$ are adjacent if and only if there exist maximal chains $A$ and $B$ in $L$ such that, set theoretically, $A\backslash B=\{t_i\}$, $B\backslash A=\{t_j\}$. Note that this is equivalent to $\rank(t_i)=\rank(t_j)=k$ and $\rank (t_i\vee t_j)=k+1$ (and/or, dually, $\rank (t_i\wedge t_j)=k-1$). 
The proof is now divided into three steps, the idea being as follows. In the first step we compute the number of connected components of $G(L)$. In the second step we introduce an oriented incidence matrix $B(G(L))$, and compute its rank using the first step. By \cite[Proposition~3.1]{binomialideals} the Krull dimension of our chain algebra equals the dimension of the $\QQ$-vector space $V=\langle\{\rho(C)\mid C \text{ is a maximal chain of }L\}\rangle$. The third step is then to compute $\dim V$, using the matrix $B(G(L))$ as a tool. 

Step 1: By construction, $G(L)$ has at least as many connected components as the number of different values the rank function takes, which is $|P|+1$. We shall prove that $t_i$ and $t_j$ belong to the same connected component of $G(L)$ if $\rank(t_i)=\rank(t_j)$, giving $G(L)$ exactly $|P|+1$ components. The proof is by induction on the rank. The base case $k=0$ is clear since we only have one element of rank $0$. Now assume the statement holds for all ranks up to $k$, and let $\rank{t_i}=\rank{t_j}=k+1$. Let $t_{i'}$ be some element covered by $t_i$ and let $t_{j'}$ be some element covered by $t_j$. Note that such elements will always exist and have rank $k$. By the induction hypothesis, $t_{i'}$ and $t_{j'}$ belong to the same connected component of $G(L)$. In other words, there is a path $t_{i'}t_{a_1}\ldots t_{a_p}t_{j'}$ in $G(L)$. Assume that this is the shortest possible path between $t_{i'}$ and $t_{j'}$. We proceed by induction on the length of this path. \Cref{rank} illustrates the idea. The statement clearly holds in the base case: if our path has length $0$, that is to say, if $i'=j'$, then $t_{i'}=t_{j'}$ is covered by both $t_i$ and $t_j$ therefore, $t_i$ and $t_j$ are either equal or adjacent in $G(L)$. For the induction step, let $t_{b_1}:=t_{i'}\vee t_{a_1}$. Since $t_{i'}$ and $t_{a_1}$ are adjacent in $G(L)$, we get $\rank (t_{b_1})=k+1$, and thus $t_{b_1}$ covers both $t_{i'}$ and $t_{a_1}$. Moreover, $t_i$ and $t_{b_1}$ either coincide or are adjacent in $G(L)$. Indeed, they both cover $t_{i'}$, which means $t_{i'}=t_i\wedge t_{b_1}$, unless $i=b_1$. Now that we have shown that $t_i$ is either adjacent or equal to $t_{b_1}$, it is enough to prove that $t_{b_1}$ is in the same connected component as $t_j$. We already know that $t_{b_1}$ covers $t_{a_1}$ and $t_j$ covers $t_{j'}$, and that $t_{a_1}\ldots t_{j'}$ is the shortest possible path between $t_{a_1}$ and $t_{j'}$, thus the statement holds by induction.

\begin{figure}
\centering
\begin{tikzpicture}[scale=1.3]

\node [above, thin] at (-1.5,0) {rank $k+1$};
\node [below, thin] at (-1.5,-1) {rank $k$};
\node [above, thin] at (1,0) {$t_i$};
\node [below, thin] at (1.5,-1) {$t_{i'}$};
\node [above, thin] at (2,0) {$t_{b_1}$};
\node [below, thin] at (2.5,-1) {$t_{a_1}$};
\node [above, thin] at (3,0) {$\ldots$};
\node [below, thin] at (3.5,-1) {$\ldots$};
\node [above, thin] at (4,0) {$t_{b_p}$};
\node [below, thin] at (4.5,-1) {$t_{a_p}$};
\node [above, thin] at (5,0) {$t_{b_{p+1}}$};
\node [below, thin] at (5.5,-1) {$t_{j'}$};
\node [above, thin] at (6,0) {$t_j$};

\draw[thick] (1,0)--(1.5,-1)--(2,0)--(2.5,-1);
\draw[dashed] (2.5,-1)--(3,0)--(3.5,-1)--(4,0);
\draw[thick] (4,0)--(4.5,-1)--(5,0)--(5.5,-1)--(6,0);

\end{tikzpicture}
\caption{Illustration of step 1 in the proof of \Cref{thm:dim}}
\label{rank}
\end{figure}
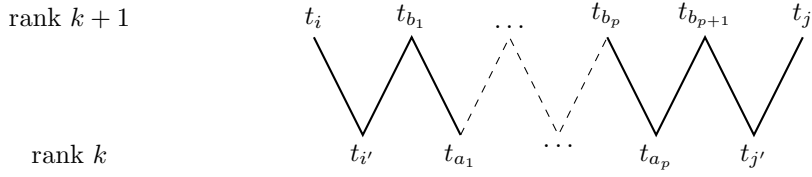

Step 2: An oriented incidence matrix $B(G(L))$ is constructed as follows. We give some direction to the edges of $G(L)$; the columns of $B(G(L))$ are indexed by edges, and the rows are indexed by vertices. The column of $t_i \to t_j $ gets a $-1$ in row $i$, a $1$ in row $j$, and $0$ in all other rows.
By \cite[Proposition 4.3]{Biggs} the rank of the incidence matrix of a directed graph equals the number of vertices minus the number of connected components of the graph. Therefore, by Step~1 we get $\rank B(G(L))=|L|-|P|-1$. We let $U$ denote the column space of $B(G(L))$. Note that the entries of a vector in $U$ always sum up to $0$. 

Step 3: Let $C$ be any maximal chain in $L$. The vector $\rho(C)$ does not belong to $U$ since its entries are non-negative, not all zero, and hence do not sum up to $0$. Let $V'=\langle U,\rho(C)\rangle$. Then $\dim V'=\dim U+1=|L|-|P|$. We shall now prove that $V'=V$, where $V=\langle\{\rho(C)\mid C \text{ is a maximal chain of }L\}\rangle$.
Let us start with the inclusion $V'\subseteq V$. Recall that $U$ is by definition spanned by columns of $B(G(L))$, and each column corresponds to an edge connecting a pair of vertices $t_i$ and $t_j$ such that $A\setminus B=\{t_i\}$, $B\setminus A=\{t_j\}$ for some chains $A$ and $B$. Then the column vector corresponding to the edge $t_i \to t_j$ equals $\rho(B)-\rho(A)$.
Therefore, $U\subseteq V$, and as $\rho(C)\in V$ we have $V'\subseteq V$.

We move on to proving the inclusion $V \subseteq V'$. It is sufficient to show that $\rho(D)\in V'$ for any maximal chain $D$ in $L$. In fact, it is sufficient to show that for any such $D$ we have $\rho(D)-\rho(C)\in U$. Note that the entries of $\rho(D)-\rho(C)$ belong to the set $\{0,1,-1\}$ and sum up to $0$. If $\rho(D)-\rho(C)=0$, we are done. Otherwise, there exists an index $i$ such that the $i$th entry of $\rho(D)-\rho(C)$ is $1$. Then $t_i$ is an element of $D$, but not $C$. Let $k=\rank (t_i)$. Since every chain contains a unique element of each rank, there exists $t_j\in C$, $i\not=j$, such that $\rank(t_j)=k$. Then clearly $t_j\not\in D$ since $D$ already possesses an element of rank $k$, namely, $t_i$. In other words, the $i$th and the $j$th entries of $\rho(D)-\rho(C)$ are $1$ and $-1$, respectively. By Step~1 we know that $t_i$ and $t_j$ (having the same rank) belong to the same connected component of $G(L)$ and therefore $e_i-e_j\in U$. It is enough to show $(\rho(D)-\rho(C))-(e_i-e_j)\in U$, and we proceed by induction on the number of entries $1$ of $\rho(D)-\rho(C)$.

We have now proved that $V'=V$, and we can conclude that $\dim K[\C_L]= \dim V = \dim V' = |L|-|P|$.
\end{proof}

Next, we discuss the degrees of minimal generators of the toric ideal of a chain algebra.

Given an FDL $L$ and a positive integer $a$, consider the bipartite graph on the rank $a$ and rank $a+1$ elements if $L$. The edges of this graph are pairs $I \precdot I'$, where $I$ has rank $a$.

\begin{Theorem}\label{thm:cycle}
Let $L$ be an FDL, such that the bipartite graph on the rank $a$ and $a+1$ elements has an induced cycle of length $2s$, for some positive integer $a$ and $s>2$. Then the defining ideal of the chain algebra $K[\C_L]$ has a minimal generator of degree $s$. 
\end{Theorem}
\begin{proof}
Let $t_1, \ldots, t_s$ and $t_{1,2}, t_{2,3}, \ldots, t_{s-1,s}, t_{s,1}$ be labels of the elements of rank $a$ and $a+1$ in $L$ that constitutes an induced cycle $\mathscr{C}$
\begin{equation}\label{eq:induced_cycle}
t_1 \precdot t_{1,2}  \succdot t_2 \precdot t_{2,3} \succdot t_3 \precdot \   \cdots \  \succdot t_s \precdot t_{s,1} \succdot t_1
\end{equation}
of length $2s$. Take $s$ maximal chains $C_1, \ldots, C_s$ in $L$ such that $C_i=\{c_1^{(i)}\precdot \dots \precdot c_r^{(i)} \}$ with $c_a^{(i)}=t_i$ and $c_{a+1}^{(i)}=t_{i,i+1}$ where $t_{s,s+1}$ should be understood as $t_{s,1}$.
In addition, take another $s$ maximal chains $\widehat{C}_1, \ldots, \widehat{C}_s$ in $L$ such that 
\[
\widehat{C}_i= \{c_1^{(i)} \precdot \dots \precdot c_a^{(i)} \precdot c_{a+1}^{(i-1)} \precdot \dots \precdot c_{r}^{(i-1)} \}
\]
where $c_j^{(0)}$ should be read as $c_j^{(s)}$. 
Note that there are precisely two perfect matchings of the cycle $\mathscr{C}$, and those are given by $\{c_a^{(i)}\precdot c_{a+1}^{(i)}\}_{i=1, \ldots, s}$ and $\{c_a^{(i)} \precdot c_{a+1}^{(i-1)}\}_{i=1, \ldots, s}$. The two matchings have no edge in common. 

The monomial identity $t_{C_1}\! \cdots t_{C_s} = t_{\widehat{C}_1}\! \cdots t_{\widehat{C}_s}$ gives rise to a binomial of degree $s$ in the defining ideal of $K[\C_L]$. If this binomial is generated by binomials of lower degrees, then by \cite[Lemma~3.8]{binomialideals}  there is a sequence of $s$-tuples of chains 
\[
\begin{pmatrix} C_1\\ \vdots \\ C_s \end{pmatrix} = \begin{pmatrix} E_1^{(0)} \\ \vdots \\ E_s^{(0)} \end{pmatrix} \rightsquigarrow \begin{pmatrix} E_1^{(1)} \\ \vdots \\ E_s^{(1)} \end{pmatrix} \rightsquigarrow \dots \rightsquigarrow \begin{pmatrix} E_1^{(j)} \\ \vdots \\ E_s^{(j)} \end{pmatrix} =   \begin{pmatrix} \widehat{C}_1\\ \vdots \\ \widehat{C}_s \end{pmatrix}
\]
all defining identical monomials $t_{E_1^{(i)}} \cdots t_{E_s^{(i)}}$, and at least one chain fixed in each step. For each of the $s$-tuples, the rank $a$ and $a+1$ pieces must be a perfect matching of the induced subgraph with vertex set $\{t_1, \ldots, t_s, t_{1,2}, \ldots, t_{s-1,s},t_{s,1}\}$, which is the cycle $\mathscr{C}$. As $(C_1, \ldots, C_s)$ and $(\widehat{C}_1, \ldots, \widehat{C}_s)$ correspond to the two possible matchings of $\mathscr{C}$ there must be a step $(E_1^{(i)}, \ldots, E_s^{(i)}) \rightsquigarrow (E_1^{(i+1)}, \ldots, E_s^{(i+1)})$ where we switch from one matching to the other. But this contradicts the fact that for every $i$ there exists $k_i$ such that $E_{k_i}^{(i)}=E_{k_i}^{(i+1)}$. We conclude that the defining ideal of $K[\C_L]$ contains a binomial of degree $s$ which is not generated by the binomials of lower degrees. 
\end{proof}

\begin{Corollary}
    
\label{thm:deg_n_rel}
 Let $P$ be a poset of width $n>2$, and let $L=\J(P)$. Then the defining ideal of the chain algebra $K[\C_L]$ has a minimal generator of degree $n$. 
\end{Corollary}
\begin{proof}
    As $P$ has width $n>2$ there is an antichain $p_1, \ldots, p_n$. Consider the ideal  $I=(p_1,\ldots, p_n)\setminus \{p_1,\ldots, p_n\}$, and let $t_i$ be the label of the ideal $I \cup \{p_i\}$, for $i=1, \ldots, n$. We note the ideals $I \cup \{p_i\}$ all have the same rank, say $a$. Let $t_{i,j}$ be the label of the rank $a+1$ ideal $I \cup \{p_i, p_j\}$. Then \eqref{eq:induced_cycle} is an induced cycle of length $2n$, so by \Cref{thm:cycle} we have a minimal generator of degree $n$.   
\end{proof}

\begin{Example}\label{ex:Bn}
If $P$ is an antichain of $n$ elements, then $\J(P)$ is the Boolean lattice $\B_n$. For $n=3,4$ the toric ideal defining the chain algebra of $\B_n$ is generated by binomials of degrees up to $n$, but when $n=5$ there is a minimal generator of degree six. To see this, let us identify the elements of $\B_5$ with subsets of $\{1,2,3,4,5\}$. An induced cycle of length 12 is given in \Cref{fig:induced_cycle}. By \Cref{thm:cycle}, with $a=2$, the defining ideal of the chain algebra has a minimal generator of degree six.
\begin{figure}

\begin{tikzpicture}
\coordinate (12) at (0,0);
\coordinate (13) at (2,0);
\coordinate (14) at (4,0);
\coordinate (45) at (6,0);
\coordinate (35) at (8,0);
\coordinate (25) at (10,0);
\coordinate (123) at (0,1);
\coordinate (134) at (2,1);
\coordinate (145) at (4,1);
\coordinate (345) at (6,1);
\coordinate (235) at (8,1);
\coordinate (125) at (10,1);

\draw [thick]    (12) -- (123)--(13)--(134)--(14)--(145)--(45)--(345)--(35)--(235)--(25)--(125)--(12);

\node [below] at (12) {$\{1,2\}$};
\node [below] at (13) {$\{1,3\}$};
\node [below] at (14) {$\{1,4\}$};
\node [below] at (45) {$\{4,5\}$};
\node [below] at (35) {$\{3,5\}$};
\node [below] at (25) {$\{2,5\}$};
\node [above] at (123) {$\{1,2,3\}$};
\node [above] at (134) {$\{1,3,4\}$};
\node [above] at (145) {$\{1,4,5\}$};
\node [above] at (345) {$\{3,4,5\}$};
\node [above] at (235) {$\{2,3,5\}$};
\node [above] at (125) {$\{1,2,5\}$};

\fill[fill=black] (12) circle (0.05 cm);
\fill[fill=black] (13) circle (0.05 cm);
\fill[fill=black] (14) circle (0.05 cm);
\fill[fill=black] (45) circle (0.05 cm);
\fill[fill=black] (35) circle (0.05 cm);
\fill[fill=black] (25) circle (0.05 cm);

\fill[fill=black] (123) circle (0.05 cm);
\fill[fill=black] (134) circle (0.05 cm);
\fill[fill=black] (145) circle (0.05 cm);
\fill[fill=black] (345) circle (0.05 cm);
\fill[fill=black] (235) circle (0.05 cm);
\fill[fill=black] (125) circle (0.05 cm);
\end{tikzpicture} 
\caption{An induced cycle of length $12$ in the Boolean lattice $\B_5$.}
\label{fig:induced_cycle}
\end{figure}
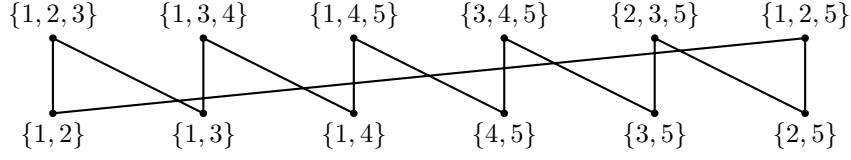
\end{Example}

We have just seen that the width of the poset $P$ does not give an upper bound for the degrees of minimal generators of the toric ideal defining the chain algebra. However, we do get a bound in terms of the number of chains, applying the following proposition. 

\begin{Proposition}\label{prop:upper_deg_bound}
Let $I \subset K[x_1, \ldots, x_n]$ be a homogeneous binomial prime ideal whose initial ideal is squarefree w.\,r.\,t.\ any DegRevLex order. Then any minimal generator of $I$ has degree at most $n/2$. 
\end{Proposition}
\begin{proof}
    For a monomial $m \in K[x_1, \ldots, x_n]$, let $\supp(m)$ denote its support, that is the set of variables which divide $m$. 

    Let $f_1=x_1^{\alpha_1} \cdots x_n^{\alpha_n}-c_1x_1^{\beta_1}\cdots x_n^{\beta_n}$ be a minimal generator of $I$. As $I$ is prime the two terms of $f_1$ have no variable in common. Consider a DegRevLex order $<$ where the variables in $\supp(x_1^{\alpha_1} \cdots x_n^{\alpha_n})$ are greater than any other variable, and the variables in $\supp(x_1^{\beta_1} \cdots x_n^{\beta_n})$ are smaller than any other variable. Then the leading term of $f$ is $x_1^{\alpha_1} \cdots x_n^{\alpha_n}$. If this term is not squarefree, there is a binomial $g=\underline{x_1^{\gamma_1} \cdots x_n^{\gamma_n}}-c_2x_1^{\delta_1} \cdots x_n^{\delta_n}$ in the Gröbner basis of $I$ w.\,r.\,t.\ $<$ with squarefree leading term dividing the leading term of $f$. We reduce $f$ modulo $g$ and get a new binomial
    \[
f_2=    f_1-mg= c_2mx_1^{\delta_1} \cdots x_n^{\delta_n}-c_1x_1^{\beta_1} \cdots x_n^{\beta_n}
    \]
     where $m=x_1^{\alpha_1-\gamma_1} \cdots x_n^{\alpha_n-\gamma_n}$. We have $\supp(x_1^{\delta_1} \cdots x_n^{\delta_n}) \cap \supp (x_1^{\beta_1} \cdots x_n^{\beta_n})= \emptyset$, as otherwise $f_1$ would be generated by binomials in $I$ of lower degree, contradicting the assumption that $f_1$ is a minimal generator. By definition of the monomial order $<$ the leading term of $f_2$ is   $c_2mx_1^{\delta_1} \cdots x_n^{\delta_n}$, which is smaller than the leading term of $f_1$. If the leading term of $f_2$ is not squarefree we continue and reduce $f_2$ in the same way. After a finite number of reductions we must arrive at a binomial $f_3=m'-c_1x_1^{\beta_1} \cdots x_n^{\beta_n}$ in $I$ with the same degree as $f_1$, whose leading term $m'$ is squarefree. Next, let's choose another DegRevLex order $<'$ where the variables in $\supp(x_1^{\beta_1} \cdots x_n^{\beta_n})$ are the largest and the variables of $\supp(m')$ are the smallest. Applying the same procedure as above but w.\,r.\,t.\ $<'$ results in a binomial $f_4=m''-m'$ where $\supp(m'') \cap \supp(m')=\emptyset$ and both terms are squarefree. Then $f_4$ can have degree at most $n/2$, and as $\deg(f_1)=\deg(f_4)$ we are done. 
\end{proof}

As the defining toric ideal of a chain algebra has a squarefree inital ideal w.\,r.\,t.\ any DegRevLex order by Proposition \ref{prop:normal_CM}, we get the following. 

\begin{Corollary}\label{cor:upper_deg_bound}
Let $L$ be an FDL, and let $I$ be the defining toric ideal of the chain algebra $K[\C_L]$. Then any minimal generator of $I$ has degree at most half the number of maximal chains in $L$. 
\end{Corollary}

For the Boolean lattice $\mathcal{B}_5$ in Example \ref{ex:Bn}, Corollary \ref{cor:upper_deg_bound} gives an upper bound of $5!/2=30$ for the degrees of minimal generators of the defining ideal. It is not known whether generators of this degree are actually needed.  

\begin{Problem}
Describe a minimal generating set for the defining ideal of the chain algebra of a finite distributive lattice. What are the degrees of the generators?
\end{Problem}

\section{Chain algebras of planar distributive lattices}\label{sec:planar}

Recall that an FDL $L$ is called planar if $\dim (L)\le 2$. Planarity is characterised by several equivalent conditions, see \cite{chenkoh}. 
Informally, a lattice being planar means that its Hasse diagram can be drawn on a plane without self-crossings. 
For us, it will be useful that $L$ can be embedded into a square grid, as explained below. 
Note that if $\dim(L)$ is $0$ or $1$ the chain algebra is isomorphic to the polynomial ring in one variable.
To avoid technicalities we assume $\dim(L)=2$ throughout this section, noting that all results are trivially true in the case $\dim(L)<2$.

\subsection{Algebraic properties of planar chain algebras}\label{subsec:alg_prop_planar}
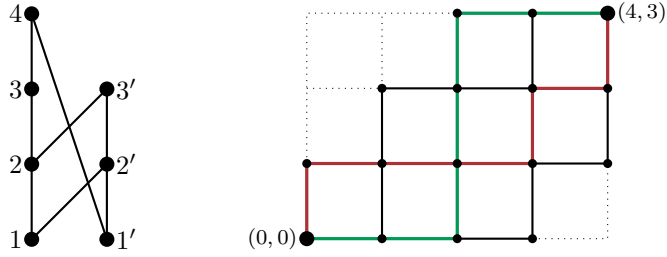
\begin{figure}
\centering

\begin{tikzpicture}

\coordinate (1) at (0,0);
\coordinate (2) at (0,1);
\coordinate (3) at (0,2);
\coordinate (4) at (0,3);
\coordinate (1') at (1,0);
\coordinate (2') at (1,1);
\coordinate (3') at (1,2);

\draw [thick]    (1) -- (2)--(3)--(4);
\draw [thick]     (1') -- (2')--(3');
\draw [thick]  (1')--(4);
\draw [thick]  (1)--(2');
\draw [thick]  (2)--(3');

\node [left] at (1) {$1$};
\node [left] at (2) {$2$};
\node [left] at (3) {$3$};
\node [right] at (1') {$1'$};
\node [right] at (2') {$2'$};
\node [right] at (3') {$3'$};
\node [left] at (4) {$4$};

\fill[fill=black] (1) circle (0.1 cm);
\fill[fill=black] (2) circle (0.1 cm);
\fill[fill=black] (3) circle (0.1 cm);
\fill[fill=black] (1') circle (0.1 cm);
\fill[fill=black] (2') circle (0.1 cm);
\fill[fill=black] (3') circle (0.1 cm);
\fill[fill=black] (4) circle (0.1 cm);

\end{tikzpicture}
\hspace{1cm}
\begin{tikzpicture}

\draw[thick] (0,0) -- (3,0);
\draw[thick] (0,1) -- (4,1);
\draw[thick] (1,2) -- (4,2);
\draw[thick] (2,3) -- (4,3);
\draw[thick] (0,0) -- (0,1);
\draw[thick] (1,0) -- (1,2);
\draw[thick] (2,0) -- (2,3);
\draw[thick] (3,0) -- (3,3);
\draw[thick] (4,1) -- (4,3);

\draw[dotted] (3,0) -- (4,0)--(4,1);
\draw[dotted] (0,1) -- (0,3)--(2,3);
\draw[dotted] (0,2) -- (1,2)--(1,3);

\draw[very thick, Maroon] (0,0)--(0,1)--(1,1)--(2,1)--(3,1)--(3,2)--(4,2)--(4,3);
\draw[very thick, ForestGreen] (0,0)--(1,0)--(2,0)--(2,1)--(2,2)--(2,3)--(3,3)--(4,3);

\fill[fill=black] (0,0) circle (0.1 cm);
\fill[fill=black] (0,1) circle (0.06 cm);
\fill[fill=black] (1,0) circle (0.06 cm);
\fill[fill=black] (1,1) circle (0.06 cm);
\fill[fill=black] (1,2) circle (0.06 cm);
\fill[fill=black] (2,0) circle (0.06 cm);
\fill[fill=black] (2,1) circle (0.06 cm);
\fill[fill=black] (2,2) circle (0.06 cm);
\fill[fill=black] (2,3) circle (0.06 cm);
\fill[fill=black] (3,0) circle (0.06 cm);
\fill[fill=black] (3,1) circle (0.06 cm);
\fill[fill=black] (3,2) circle (0.06 cm);
\fill[fill=black] (3,3) circle (0.06 cm);
\fill[fill=black] (4,1) circle (0.06 cm);
\fill[fill=black] (4,2) circle (0.06 cm);
\fill[fill=black] (4,3) circle (0.1 cm);

\node [below, left] at (0,0) {\footnotesize $(0,0)$};
\node [above, right] at (4,3) {\footnotesize $(4,3)$};

\end{tikzpicture}
\caption{A poset $P$, the lattice $L=\J(P)$ embedded into a grid, and a pair of its incomparable maximal chains}
\label{fig:Grid_example}
\end{figure}

For a planar FDL $L=\J(P)$, the underlying poset $P$ is covered by two chains $A=\{1\prec 2 \prec \dots \prec a\}$ and $B=\{1'\prec  2'\prec   \ldots \prec b'\}$. An ideal $I$ of $P$ with $|I\cap A|=i$ and $|I\cap B|=j$ is assigned the label $t_{ij}$. It is easily verified that this assignment is injective.
We now have a natural embedding of $L$ into $\NN^2$, with the convention that $0\in \NN$. The maximal chains of $L$ are depicted as paths taking steps $(i,j) \to (i+1,j)$ and $(i,j) \to (i,j+1)$, with starting point $(0,0)$ and end point $(a,b)$.

With this notation, $t_{ij} \preceq t_{k \ell}$ if and only if $i \le k$ and $j \le \ell$, and the rank of $t_{ij}$ is $i+j$. We extend this to a total order on the $t_{ij}$'s by saying that $t_{ij} <t_{k \ell}$ if $i+j < k + \ell$ or $i+j=k+ \ell$ and $i <k$.
Next, we define a partial order $\preceq_{\operatorname{ch}}$ on the set of maximal chains of $L$. 
Let $C=\{t_{00}=c_0 \precdot c_1\precdot\ldots\precdot c_r=t_{ab}\}$ and $D=\{t_{00}=d_0 \precdot d_1\precdot\ldots\precdot d_r=t_{ab}\}$, where $c_k$ and $d_k$ are some $t_{ij}$'s of rank $k$.
 We say that $C\preceq_{\operatorname{ch}} D$ if for all $k=1,\ldots, r$ we have $c_k\le d_k$. Considering $C$ and $D$ as $\NN^2$-paths, this can be visualised as $C$ lying weakly above and to the left of $D$, that is $C$ and $D$ are \textit{non-traversing}.
 
 Now,  suppose $C$ and $D$ are incomparable, and let $k<r$ be the largest rank such that the truncated chains $\{c_0 \precdot c_1\precdot\ldots\precdot c_k\}$ and $\{d_0 \precdot d_1\precdot\ldots\precdot d_k\}$ are comparable with respect to the partial order above. Then $c_k=d_k$ and this point will be called \textit{the first traversing point} of $C$ and $D$.
 
\begin{Example}
\label{ex:gridchains}
Let $P$ be the poset in \Cref{fig:Grid_example}. This figure also shows the corresponding embedding of $L=\J(P)$ into $\NN^2$. 
    An example of two incomparable maximal chains $\color{Maroon}{C}$ and $\color{ForestGreen}{D}$ of the lattice $L$ in \Cref{fig:Grid_example} is given by
    \[
    {\color{Maroon}C}=\{t_{00} \precdot t_{01}\precdot t_{11}\precdot t_{21}\precdot t_{31}\precdot t_{32}\precdot t_{42} \precdot t_{43}\}\]
    and 
    \[
    {\color{ForestGreen}D}=\{t_{00} \precdot t_{10}\precdot t_{20}\precdot t_{21}\precdot t_{22}\precdot t_{23}\precdot t_{33} \precdot t_{43}\}.
    \]
    The first traversing point of ${\color{Maroon}C}$ and ${\color{ForestGreen}D}$ is $(2,1)$. Following ${\color{Maroon}C}$ until the point $(2,1)$ and then following ${\color{ForestGreen}D}$ produces the greatest path smaller than both $\color{Maroon}{C}$ and ${\color{ForestGreen}D}$ w.\,r.\,t.\ the partial order $\prec_{\operatorname{ch}}$. In other words
    \[
    {\color{Maroon}C} \wedge {\color{ForestGreen}D}=\{ t_{00} \precdot t_{01}\precdot t_{11}\precdot t_{21}\precdot t_{22}\precdot t_{23}\precdot t_{33} \precdot t_{43} \},\]
    and similarly
    \[{\color{Maroon}C} \vee {\color{ForestGreen}D} =\{ t_{00} \precdot t_{10}\precdot t_{20}\precdot t_{21}\precdot t_{31}\precdot t_{32}\precdot t_{42} \precdot t_{43} \}.
    \]
\end{Example}

Taking the meet and join of maximal chains can also be realised as a sorting of monomials. The sorting is performed using the total order on variables $t_{ij}$ defined above, from smaller to larger. In Example \ref{ex:gridchains} we get $\sort(t_{\color{Maroon}C},t_{\color{ForestGreen}D}) = (t_{{\color{Maroon}C} \wedge {\color{ForestGreen}D}},t_{{\color{Maroon}C} \vee {\color{ForestGreen}D}})$.

\begin{Lemma}\label{lemma:planar-sort-hibi}
Let $L$ be a planar FDL whose elements are labelled as described above. The set of maximal chains of $L$ constitutes an FDL $L'$, with respect to the partial order $\preceq_{\operatorname{ch}}$. For two maximal chains $C=\{c_0\precdot\ldots\precdot c_r\}$ and $D=\{d_0\precdot\ldots\precdot d_r\}$ in $L$, their meet and join in $L'$ are given by 
\begin{align*}
&C \wedge D = \{ \min(c_0, d_0)\precdot \dots \precdot \min(c_i, d_i) \precdot \dots \precdot \min(c_r,d_r) \} \\ 
&C \vee D = \{ \max(c_0, d_0) \precdot \dots \precdot \max(c_i, d_i) \precdot \dots \precdot \max(c_r,d_r)\}. 
\end{align*}
Moreover, we have $\sort(t_C,t_D) = (t_{C \wedge D}, t_{C \vee D})$.
\end{Lemma}
\begin{proof}
 Let's first consider the case $C\preceq_{\operatorname{ch}} D$. Here we have $c_k\le d_k$ for all $k=0,\ldots, r$. We also have $d_k\le c_{k+1}$ for all $k=0,\ldots, r-1$ since $d_k$ has rank $k$ and $c_{k+1}$ has rank $k+1$. We conclude that 
 \[c_0=d_0 \le c_1\le d_1\le c_2\le d_2\le\ldots\le c_r= d_r\]
 and thus the sorting process on the pair $(t_C, t_D)$ returns $(t_C,t_D)$. It is clear that $C \wedge D=C$ and $C \vee D=D$.
Now let $C$ and $D$ be incomparable with respect to $\preceq_{\operatorname{ch}}$, and let $c_k=d_k$ be the first traversing point.  Without loss of generality we can assume $c_i\le d_i$ for all $0\le i\le k$, and $c_{k+1}>d_{k+1}$. 
Now, let 
\begin{align*} &C'=\{c_0\precdot c_1\precdot\ldots\precdot c_k=d_k\precdot d_{k+1}\precdot\ldots\precdot d_r\}, \quad \text{and} \\ 
&D'=\{d_0\precdot d_1\precdot\ldots\precdot d_k=c_k\precdot c_{k+1}\precdot\ldots\precdot c_r\}.\end{align*}
In other words, $C'$ travels along $C$ until rank $k$, where $C$ and $D$ traverse, and continues along $D$, while $D'$ travels along $D$ until rank $k$ and continues along $C$. Clearly $C'$ and $D'$ are chains of $L$, and $\sort(t_C,t_D)=\sort(t_{C'},t_{D'})$. If $C' \preceq_{\operatorname{ch}} D'$, the sorting is complete. Otherwise there is a rank $k'>k$ where $C'$ and $D'$ first traverse, and we repeat the argument. Eventually we arrive at $\sort(t_C, t_D)=(t_{\widetilde C}, t_{\widetilde D})$. By construction
\begin{align*}
\widetilde{C} =& \{ \min(c_0, d_0)\precdot \dots \precdot \min(c_i, d_i) \precdot \dots \precdot \min(c_r,d_r) \} \\ 
\widetilde{D} =&  \{ \max(c_0, d_0) \precdot \dots \precdot \max(c_i, d_i) \precdot \dots \precdot \max(c_r,d_r)\}. 
\end{align*}
Moreover, $\widetilde C$ is the largest chain which is smaller than both $C$ and $D$, so $\widetilde{C}=C\wedge D$, and similarly $\widetilde{D}= C \vee D$. This proves the claim on the sorting, and that $L'$ is a lattice with the prescribed meet and join. We will see in Lemma \ref{lemma:youngs_lattice} that $L'$ can in fact be identified with an interval of Young's lattice and is therefore distributive.
\end{proof}

Finally, we would like to introduce the notion of a sortable FDL. We have seen that sortable sets of monomials define toric algebras with well understood defining ideals. We have also seen that the meet and join of maximal chains of a planar FDL can be expressed via the sorting operator. This leads to the following definition:
\begin{Definition}
Let $L=\mathcal{J}(P)$, as before. We will say that $L$ is \textit{sortable} if there exists a labelling on $L$ such that $B=\{t_C\mid C \text{ is a maximal chain of } L\}$ is a sortable set of monomials.
\end{Definition}

Applying Lemma \ref{lemma:planar-sort-hibi} we can now give a classification of the chain algebras arising from planar FDLs.

\begin{Theorem}\label{thm:mainplanar}
The following are equivalent for an FDL $L$.

\begin{enumerate}
    \item $L$ is planar,
    \item $L$ is sortable,
    \item The defining ideal of $K[\mathcal{C}_L]$ has a quadratic Gröbner basis with respect to DegRevLex,
    \item The defining ideal of $K[\mathcal{C}_L]$ is quadratically generated,
    \item $K[\C_L]$ is Koszul,
    \item $K[\C_L]$ is a Hibi ring.

\end{enumerate}
\end{Theorem}

\begin{proof}
We will first prove ${\it (1)}\Rightarrow {\it (2)}\Rightarrow {\it (3)}\Rightarrow {\it (5)}\Rightarrow {\it (4)}\Rightarrow {\it (1)}$ and then connect these five statements to {\it (6)}.

The implication ${\it (1)} \Rightarrow {\it (2)}$ follows by \Cref{lemma:planar-sort-hibi}


${\it (2)} \Rightarrow {\it (3)}$: let us assume that $L$ has $m$ maximal chains $C_1, \ldots, C_m$, indexed such that if $C_i \preceq_{\operatorname{ch}} C_j$ then $i<j$.  We then define a $K$-algebra homomorphism $ \phi: K[T_1,\ldots, T_m] \to K[\mathcal{C}_L]$  by $T_i\mapsto t_{C_i}$ for $i=1,\ldots, m$. On the polynomial ring $K[T_1,\ldots, T_m]$ we impose the DegRevLex monomial order with $T_1 > \dots > T_m$. 
By \Cref{thm:order}, it is then enough to show that this is a sorting order. For two maximal chains $C$ and $D$ the equality $\sort(t_C,t_D)=(t_{C \wedge D},t_{C \vee D})$ lifts to the binomial $T_{i}T_{j}-T_kT_{\ell} \in \ker \phi$ where $T_{i}\mapsto t_{C}$, $T_{j}\mapsto t_{D}$, $T_{k}\mapsto t_{C \wedge D}$, $T_{\ell}\mapsto t_{C \vee D}$. Then $T_jT_j$ is the leading term with respect to DegRevLex.

The implication ${\it (3)} \Rightarrow {\it (5)}$ was proved in \cite{Anick}, and it is well known that Koszul algebras are quadratic so ${\it(5)} \Rightarrow {\it(4)}$.

 We saw in \Cref{thm:deg_n_rel} that if the lattice $L$ is not planar, then the defining binomial ideal is not quadratic. Hence ${\it (4)} \Rightarrow {\it (1)}$.

We have now proved that the first five statements are equivalent.

Now, ${\it (1)} \Rightarrow {\it (6)}$ follows by \Cref{lemma:planar-sort-hibi}. More precisely, $K[\C_L]$ is isomorphic to the Hibi ring $K[L']$.

Finally, it follows from the definition of a Hibi ring that ${\it (6)} \Rightarrow {\it (4)}$, which completes the proof. 
\end{proof}



\subsection{Hilbert series of planar chain algebras} We will now interpret the Hilbert series of the algebra $K[\C_L]$ combinatorially. A grading on $K[\C_L]$ is defined by assigning degree one to the monomials generating $K[\C_L]$. From \Cref{lemma:planar-sort-hibi} and \Cref{thm:mainplanar} we get the following description of the Hilbert function. 

\begin{Corollary}\label{cor:Hilbert_series}
For a planar FDL $L$, the vector space dimension of the graded component $K[\C_L]_i$ equals the number of ways to draw $i$ non-traversing maximal lattice paths in $L$.
\end{Corollary}
By a \emph{maximal lattice path} we mean a path from $(0,0)$ to $(a,b)$ where the points are identified with elements of $L$, and $(a,b)$ is the maximal element, as described in Section \ref{subsec:alg_prop_planar}. We clarify that the maximal lattice paths are allowed to have points in common, only not to traverse.  

As our planar chain algebra $K[C_L]$ is isomorphic to the Hibi ring $K[L']$, where $L'$ is the FDL of chains as in \Cref{lemma:planar-sort-hibi}, the Hilbert series is given by \eqref{eq:Hilbert_series}. To analyse the numerator of \eqref{eq:Hilbert_series} in the case of chain algebras, we need to deduce the poset $P'$ for which $L'=\J(P')$. 

As before, we consider a lattice $L=\J(P)$ where $P$ is a poset covered by two chains $A=\{1\prec 2 \prec \dots \prec a\}$ and $B=\{1'\prec  2'\prec   \ldots \prec b'\}$. The elements of $L$ are represented by integer points $(i,j)$ with $0 \le i \le a$ and $0 \le j \le b$. The maximal chains of $L$, constituting the lattice $L'$, are depicted as paths in $\NN^2$ connecting the points $(0,0)$ and $(a,b)$. With the goal of determining the poset $P'$ it is preferable to consider \emph{cells} of $\NN^2$, rather than integer points. By a \emph{cell} we mean a $1 \times 1$ square where the corners are integer points. A lattice path $C$ from $(0,0)$ and $(a,b)$ determines a integer partition $\lambda=(\lambda_b, \ldots, \lambda_1)$ where $\lambda_i$ is the number of cells to the left of $C$ in row $i$, inside the rectangle with $(0,0)$ and $(a,b)$ as its bottom left and top right corner. Here \enquote{row $1$} means the bottom row of cells, and \enquote{row $b$} means the top row. The shape that $C$ cuts from the upper left corner of the rectangle is the Young diagram of the integer partition $\lambda$. We note that $\lambda_b \ge \dots \ge \lambda_1 \ge 0$, and it may happen that some $\lambda_i$'s are 0. For instance, the smallest and the largest chains of \Cref{fig:Grid_example} can be identified with the integer partitions $(2,1,0)$ and $(4,4,3)$, respectively. The chains $\color{Maroon}{C}$ and ${\color{ForestGreen}D}$ from \Cref{ex:gridchains} are identified with $(4,3,0)$ and $(2,2,2)$. For a proper introduction to integer partitions and Young diagrams we refer the reader to Chapter 1.7 of \cite{Stanley}. 

Young diagrams can be ordered by inclusion, and in this way constitute a distributive lattice, see \cite[Example 3.4.4, Example 3.21.2]{Stanley}. If $\mu \subset \lambda$ are two Young diagrams, the interval $[\mu, \lambda]$ consists of all Young diagrams containing $\mu$ and contained in $\lambda$. The interval  $[\mu, \lambda]$ is itself a distributive lattice.

In the next lemma we use the notation $(p)$, where $p \in P$, for the smallest ideal of $P$ containing $p$. 

\begin{Lemma}\label{lemma:youngs_lattice}
Let $P=A \sqcup B$, $L$ and $L'$ be as above. Then $L'$ is isomorphic to an interval $[\mu, \lambda]$ of Young's lattice. More precisely, the minimal and maximal elements of $L'$ are given by integer partitions  $\mu = (\mu_b, \ldots, \mu_1)$ and $\lambda=(\lambda_b, \ldots, \lambda_1)$ where
    \begin{align*}
    \mu_i &= |A \cap (i')|, \quad \text{and} \\
    \lambda_i &= \max \  k \in A \ \text{such that} \ i' \notin (k).
    \end{align*}
\end{Lemma}
\begin{proof}
 An integer partition $\lambda=(\lambda_b, \ldots, \lambda_1)$ corresponds to the path 
 \begin{align*}
 (0,0) \to (1,0) \to  \dots \to & (\lambda_1,0) \to (\lambda_1,1) \to  \dots \to (\lambda_2,1) \to (\lambda_2, 2) \to \cdots  \\
 & \dots \to (\lambda_i, i-1) \to (\lambda_i,i) \to \dots \to (\lambda_{i+1},i) \to \dots \\ 
 & \hspace{50 pt} \dots \to (\lambda_b,b-1) \to (\lambda_b,b) \to \dots \to (a,b). 
 \end{align*}
 It is clear that the partial order $\preceq_{\operatorname{ch}}$ on $L'$ translates to inclusion of Young diagrams, making $L'$ isomorphic to an interval of Young's lattice. 
 
 To get the maximal element of $L'$ we choose each $\lambda_i$ as large as possible, meaning that $(\lambda_i, i-1)$ is a valid point in $L$ but $(\lambda_i+1,i-1)$ is not. This occurs precisely when $i' \notin (\lambda_i)$ but $i' \in (\lambda_i+1)$, proving the claimed formula for the maximal element of $L'$. 

 To obtain the minimal element $\mu = (\mu_b, \ldots, \mu_1)$ we choose the smallest $\mu_i$ such that $(\mu_i,i)$ is a valid point in $L$. This gives $\mu_i=|A \cap (i')|$.  
\end{proof}

 Adapting the terminology from Chapter 7.10 of \cite{Stanley2}, the cells sitting between the minimal and maximal paths of $L'$ form the \emph{skew shape} $\lambda / \mu$.

\begin{Lemma}\label{lemma:P'_cells}
    Let $L'$, $\mu$, and $\lambda$ be as in Lemma \ref{lemma:youngs_lattice}. Define a poset $P'$ on the cells of the skew shape $\lambda / \mu$, by imposing strict increase when moving to the right along rows, and downwards along columns.  Then $L' \cong \J(P').$
\end{Lemma}
\begin{proof}
The smallest Young diagram of $L'$ containing a fixed cell $c$ is produced by taking all cells above and to the left of $c$ inside the minimal rectangle containing $L$. This Young diagram is join irreducible, as the only cell which can be removed, producing a new Young diagram of $L'$, is $c$.
Conversely, every Young diagram where only one cell is removable can be obtained this way. We have now seen that the elements of $P'$ are identified with the join irreducibles of $L'$, and hence $L'\cong \J(P')$.     
\end{proof}

\begin{Remark}
From the theory of Hibi rings we know that the Krull dimension of the Hibi ring $K[L']$ is given by $|P'|+1$, and as discussed above, the elements of the poset $P'$ are identified with cells of a skew shape $\lambda / \mu$. On the other hand, $K[L']\iso K[\C_L]$, and from \Cref{thm:dim} we know that $\dim(K[\C_L])=|L|-|P|$. 
Let us check that these formulas agree, i.\,e.\ that $|L|-|P|$ equals the number of cells of the skew shape $\lambda / \mu$, increased by one. First, note that the cells are in bijection with the elements of $L$ that have exactly two lower neighbours, by identifying a cell with its top right corner in the rectangular grid. 
Now, $P$ is the set of join-irreducible elements of $L$, that is, those with exactly one lower neighbour. Therefore, $|L|-|P|$ is the number of elements with zero or two lower neighbours, as two is the maximal number of neighbours in the planar case. The only element without a lower neighbour is the minimal element of $L$, and hence $|L|-|P|=|P'|+1$.
\end{Remark}

The chains of $P'$ are realised as paths between cells, taking steps down and to the right (the directions in which Young diagrams increase). In particular it follows that $K[L']$ is Gorenstein if and only if all maximal paths of cells have the same length. 

\begin{Example}
    Let $P$ and $L$ be as in \Cref{fig:Grid_example}. As previously noted, $L'$ can be identified with the interval $[(2,1,0), (4,4,3)]$ of Young's lattice. Then the poset $P'$ consists of elements $p_1, p_2, \ldots, p_8$, identified with the skew shape of the cells of $L$ as illustrated in \Cref{fig:SYT}. Denoting the partial order on $P'$ by  $\prec'$, the maximal chains of $P'$ are
    \begin{align*} &\{ p_1 \prec' p_2 \prec' p_5\}, \quad \{p_1 \prec' p_4 \prec' p_5\}, \quad \{ p_1 \prec' p_4 \prec' p_8\} , \quad \{ p_3 \prec' p_4 \prec' p_5\},\\
    & \{p_3 \prec' p_4 \prec' p_8\}, \quad \{p_3 \prec' p_7 \prec' p_8 \}, \quad \text{and} \quad  \{p_6 \prec' p_7 \prec' p_8\}.     
    \end{align*}
    As they all have the same length, $K[\C_L]$ is Gorenstein. Indeed, the Hilbert series is 
    \[
    \frac{1+18z+65z^2+65z^3+18z^4+z^5}{(1-z)^9}.
    \]
    We see that this algebra has Krull dimension $9$, which can be understood as the number of cells increased by $1$, or as $|L|-|P|=16-7$.
\end{Example}

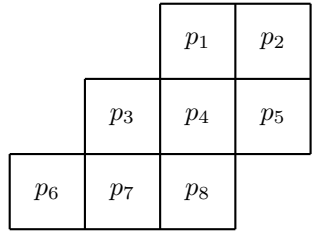
\begin{figure}
    \centering
    
\begin{tikzpicture}
\draw[thick] (0,0) -- (3,0);
\draw[thick] (0,1) -- (4,1);
\draw[thick] (1,2) -- (4,2);
\draw[thick] (2,3) -- (4,3);
\draw[thick] (0,0) -- (0,1);
\draw[thick] (1,0) -- (1,2);
\draw[thick] (2,0) -- (2,3);
\draw[thick] (3,0) -- (3,3);
\draw[thick] (4,1) -- (4,3);

\node at (2.5,2.5) {$p_1$};
\node at (3.5,2.5) {$p_2$};
\node at (1.5,1.5) {$p_3$};
\node at (2.5,1.5) {$p_4$};
\node at (3.5,1.5) {$p_5$};
\node at (0.5,0.5) {$p_6$};
\node at (1.5,0.5) {$p_7$};
\node at (2.5,0.5) {$p_8$};

\end{tikzpicture}
    \caption{A poset $P'$ is obtained by imposing strict increase along rows and columns}
    \label{fig:SYT}
\end{figure}

With the description of $P'$ as cells in a skew shape, a linear extension of $P'$ is the same as filling in the numbers $1, \ldots, |P'|$ in the cells, such that the numbers are increasing along rows and columns. In other words, a linear extension of $P'$ is a \emph{standard Young tableaux} (SYT). In a SYT $\mathcal{T}$ we say that $i$ is an \emph{ascent} if $i+1$ sits in a row above the row of $i$. We denote the number of ascents in $\mathcal{T}$ by $\operatorname{asc}(\mathcal{T})$.

\begin{Theorem}\label{thm:HS_ascents}
    Let $L$ be a planar FDL, and let $d$ be the number of cells in the corresponding skew shape $\lambda / \mu$ as in Lemma \ref{lemma:P'_cells}. The Hilbert series of the chain algebra $K[\C_L]$ is 
    \[
 \frac{\sum_{\mathcal{T}}z^{\operatorname{asc}(\mathcal{T})}}{(1-z)^{d+1}}
    \]
    where the sum is taken over all SYT's $\mathcal{T}$ of the skew shape $\lambda / \mu$. 
\end{Theorem}

\begin{proof}
    Let us index the elements  of $P'=\{p_1, \ldots, p_d\}$ so that 
    $p_1, p_2 \ldots, p_{s_1}$ are the cells in the top row, and $p_{s_1+1}, p_{s_1+2}, \ldots, p_{s_2}$ are cells of the second row, and so on, in the skew shape $\lambda / \mu$. See also \Cref{fig:SYT}. It is clear that this indexing respects the partial order on $P'$. Next, let $w$ be a linear extension of $P'$. By the expression \eqref{eq:Hilbert_series} for the Hilbert series, we are done if we can prove that the descents of the linear extension $w$ are the ascents of $w$ interpreted as SYT. So, suppose $i$ is a descent of $w$. Then we have indices $j>k$ such $w(p_j)=i$ and $w(p_k)=i+1$. From the way we indexed the cells it follows that $p_j$ either sits to the right of $p_k$ in the same row, or in a row below $p_k$. However, the assignment  $w(p_j)=i$ and $w(p_k)=i+1$ excludes the possibility of $p_j$ being to the right of $p_k$ in the same row. Then $i$ satisfies the definition of an ascent. Conversely, suppose $i$ is an ascent in the SYT given by $w$. Then there are $j$ and $k$ such that $w(p_j)=i$ and $w(p_k)=i+1$ where $p_k$ is in a row above $p_j$. By the choice of indexing $k<j$, so $i$ is  descent of $w$. 
\end{proof}

The $h$-polynomial of a Hibi ring $K[\J(P)]$, i.\,e.\ the numerator of \eqref{eq:Hilbert_series}, is the so called $(P, \omega)$-Eulerian polynomial, where $\omega$ denotes any natural labelling of the poset $P$. It was conjectured by Neggers \cite{Neggers} that these polynomials have real roots, and Stanley conjectured that the same holds even if the labelling $\omega$ is not natural. Recall that a polynomial having real roots implies that the coefficient sequence is unimodal. The Neggers-Stanley conjecture has been proved in several special cases, but in 2004 counterexamples were presented by Brändén \cite{Branden}. Later, counterexamples with natural labelling were provided by Stembridge \cite{Stem}. However, the weaker version of the conjecture, namely that $(P, \omega)$-Eulerian polynomials are unimodal, is an open problem. It is proved in \cite{RW} that the $(P, \omega)$-Eulerian polynomials are unimodal when $P$ is pure, and $\omega$ natural. Translated to Hibi rings, this means that the $h$-polynomial of a Gorenstein Hibi ring is unimodal. 

Brenti \cite{Brenti} proved the Neggers-Stanley conjecture in the case when the poset is identified with a Young diagram. So, when the minimal lattice path of $L$ is $(0,0) \to \dots \to (0,b) \to \dots \to (a,b)$, the $h$-polynomial of $K[\C_L]$ is unimodal. To the best of the authors knowledge, it is an open problem whether all the $h$-polynomials appearing in \Cref{thm:HS_ascents} are unimodal. As a final remark of this section we consider the rectangular case. 
\begin{Remark}
    If the poset $P$ is given by two chains of $\{1 \prec \dots \prec a\}$ and $\{1' \prec \dots \prec b'\}$, and no additional relations, the skew shape in Lemma \ref{lemma:P'_cells} is just an $a \times b$ rectangle. The number of SYT's with precisely $k$ ascents are counted by the so called \emph{generalised Narayana numbers} $N(a,b,k)$, see \cite[Remark 1.3]{Sulanke}. An explicit formula for $N(a,b,k)$ is given in \cite[Proposition 1]{Sulanke}. By \Cref{thm:HS_ascents} the $h$-polynomial of $K[\C_L]$ is the Narayana polynomial $N_{a,b}(z)=\sum_k N(a,b,k)z^k$. As the poset $P'$ is pure, the algebra $K[\C_L]$ is Gorenstein, and we recover the symmetry of $N_{a,b}(z)$ proved in \cite[Corollary 1]{Sulanke}.
\end{Remark}

\section{Hibi rings and sortable monomials}
It was proved in \cite[Theorem 5.3]{EHM} that certain toric rings, which generalise Hibi rings of FDLs, are isomorphic to algebras generated by sortable sets of monomials. This gives two quadratic generating sets for the defining binomial ideal of a toric Hibi ring: the binomials $x_Ix_J-x_{I\wedge J}x_{I \vee J}$ from the definition of a Hibi ring, and the \enquote{sorting binomials} \eqref{eq:sort_binomial}. 
We saw in \Cref{thm:mainplanar} that for Hibi rings isomorphic to planar chain algebras $\sort(m_I,m_J)=(m_{I\wedge J}, m_{I \vee J})$, and as a consequence the two generating sets of binomials coincide. 
In this last section we prove that this property in fact holds for any Hibi ring of an FDL. 

\begin{Theorem}
 The Hibi ring of an FDL $L$ can be realised as generated by a sortable set of monomials indexed by the elements of $L$. Moreover, for $I, J \in L$  we have $\sort(m_I,m_J)=(m_{I\wedge J}, m_{I \vee J})$.
\end{Theorem}
\begin{proof}
Let $L=\J(P)$ and let $n$ be the width of $P$. We will fix a minimal chain decomposition of $P$, that is, disjoint chains $A_1, A_2,\ldots, A_n$ such that their set theoretical union equals the underlying set of $P$. To each $I\in\mathcal{J}(P)$ we associate a vector $v_I=(k_1,\ldots, k_n) \in  \mathbb{N}^n$, where $k_i=|I\cap A_i|$ for all $i=1,\ldots,n$. In other words, the coordinates of $v_I$ tell us how many elements in $I$ belong to respective chains. We label $I$ with the monomial $m_I=x^{(1)}_{k_1}\cdots x^{(n)}_{k_n}$ in the polynomial ring $K[x^{(1)}_{0},\ldots, x^{(1)}_{|A_1|}, \dots, x^{(n)}_{0},\ldots, x^{(n)}_{|A_n|}]$ on $|P|+n$ variables divided into $n$ blocks. Each monomial $m_I$ is of degree $n$, with exactly one variable of each block. Let $v_I=(k_1,\ldots, k_n)$ and $v_J=(\ell_1,\ldots \ell_n)$. Then $v_{I\cap J}=(\min(k_1,\ell_1),\ldots, \min(k_n,\ell_n))$ and $v_{I\cup J}=(\max(k_1,\ell_1),\ldots, \max(k_n,\ell_n))$. Indeed, $I$ contains the $k_1$ \emph{smallest} elements of chain $A_1$ (otherwise $I$ is not an ideal). Note that this argument also proves the injectivity of the assignment of vectors to ideals.  Similarly $J$ contains the $\ell_1$ \emph{smallest} elements of $A_1$.  Therefore, the intersection of $I$ and $J$ contains exactly the $\min(k_1,l_1)$ smallest elements of $A_1$, and the same for $A_2, \ldots, A_n$. An analogous argument applies to $v_{I\cup J}$. Sorting the variables firstly by upper index and secondly by lower index we get
\begin{align*}
\sort(m_I,\, m_J) &= \sort(x^{(1)}_{k_1}\cdots x^{(n)}_{k_n},\ x^{(1)}_{\ell_1}\cdots x^{(n)}_{\ell_n})\\
&=(x^{(1)}_{\min(k_1,\ell_1)}\cdots x^{(n)}_{\min(k_n,\ell_n)},\ x^{(1)}_{\max(k_1,\ell_1)}\cdots x^{(n)}_{\max(k_n,\ell_n)})\\ &= (m_{I\cap J},\, m_{I\cup J}) \qedhere 
\end{align*}
\end{proof}

\begin{Example}
Let $P=\{1,2,3\}$ be the antichain on three elements. Then $L=\J(P)=\mathcal{B}_3$. The toric algebra generated by all the monomials in \Cref{fig:B3} is isomorphic to the Hibi ring of $L$, using the indexing introduced above. If we take for example $I=\{1\}$ and $J=\{2,3\}$ we get $m_I=x^{(1)}_1x^{(2)}_0x^{(3)}_0$ and $m_J=x^{(1)}_0x^{(2)}_1x^{(3)}_1$. Applying the sorting of monomials as above gives 
\[
\sort(m_I, m_J)=(x^{(1)}_0x^{(2)}_0x^{(3)}_0, \,  x^{(1)}_1x^{(2)}_1x^{(3)}_1) = (m_{I \cap J}, m_{I \cup J}).
\]
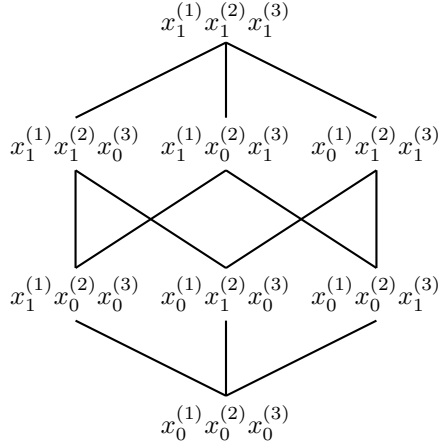
\begin{figure}
\centering
\begin{tikzpicture}

\node [below, thin] at (0,1) {$x^{(1)}_0x^{(2)}_0x^{(3)}_0$};
\node [above, thin] at (-2,2) {$x^{(1)}_1x^{(2)}_0x^{(3)}_0$};
\node [above, thin] at (0,2) {$x^{(1)}_0x^{(2)}_1x^{(3)}_0$};
\node [above, thin] at (2,2) {$x^{(1)}_0x^{(2)}_0x^{(3)}_1$};
\node [above, thin] at (-2,4) {$x^{(1)}_1x^{(2)}_1x^{(3)}_0$};
\node [above, thin] at (0,4) {$x^{(1)}_1x^{(2)}_0x^{(3)}_1$};
\node [above, thin] at (2,4) {$x^{(1)}_0x^{(2)}_1x^{(3)}_1$};
\node [thin] at (0,6) {$x^{(1)}_1x^{(2)}_1x^{(3)}_1$};

\draw[thick] (0,1)--(-2,2);
\draw[thick] (0,1)--(0,2);
\draw[thick] (0,1)--(2,2);

\draw[thick] (-2,2.7)--(-2,4);
\draw[thick] (2,2.7)--(2,4);

\draw[thick] (-2,2.7)--(0,4);
\draw[thick] (2,2.7)--(0,4);

\draw[thick] (0,2.7)--(-2,4);
\draw[thick] (0,2.7)--(2,4);

\draw[thick] (-2,4.7)--(0,5.7);
\draw[thick] (0,4.7)--(0,5.7);
\draw[thick] (2,4.7)--(0,5.7);
\end{tikzpicture}
\caption{A monomial generating set for the Hibi ring $K[\B_3]$}
\label{fig:B3}
\end{figure}
\end{Example}

We remark that not every algebra generated by a sortable set of monomials is a Hibi ring, as shown in the next example.

\begin{Example}
Consider the squarefree Veronese subalgebra of $K[x,y,z,w]$ generated by all squarefree monomials of degree two. We have 
 \[A=K[xy,xz,xw,yz,yw,zw]\iso K[T_1,\ldots, T_6]/(T_1T_6-T_2T_5, \, T_3T_4-T_2T_5)\]
 and the set of monomials is clearly sortable. 
 This algebra has Krull dimension $4$, so if it is a Hibi ring of  a lattice $\J(P)$, then $|P|=3$. There are only five non-isomorphic $3$-element posets, but none of these five Hibi rings have the same Hilbert series as $A$.
 \end{Example}

\section*{Acknowledgement}
First of all we would like to thank Jürgen Herzog for introducing the notion of chain algebras of finite distributive lattices, for suggesting the statement of \Cref{thm:mainplanar}, and for many helpful comments throughout our work with this project. Many thanks to Per Alexandersson for his insightful remarks about $(P, \omega)$-Eulerian polynomials. 
We also thank Aldo Conca for the rewarding discussions around the topics of this paper. Finally, we thank the anonymous referees for their careful reading and suggestions. 

\medskip

The first author was supported by a fellowship from the Wenner-Gren Foundations (grant WGF2022-0052), and the second author was supported by the grant KAW-2019.0512 from the Knut and Alice Wallenberg Foundation.

\medskip 

\bibliographystyle{plain}
\bibliography{Poset_references}

\end{document}